\pdfoutput=1
\documentclass[a4paper, reqno, 12pt, toc]{amsart}
\usepackage[utf8]{inputenc}
\usepackage[T1]{fontenc}
\usepackage{lmodern, csquotes, graphicx}
\usepackage{amssymb, physics, bm}
\usepackage{hyperref}
\usepackage[subsetnonamb, bbsets]{jkmath}
\usepackage[backend=biber,style=alphabetic,sorting=nyt]{biblatex}
\usepackage{dynkin-diagrams}

\newtheorem{theorem}{Theorem}[section]
\newtheorem{definition}[theorem]{Definition}
\newtheorem{lemma}[theorem]{Lemma}
\newtheorem{proposition}[theorem]{Proposition}

\newtheorem{example}[theorem]{Example}

\newtheorem{remark}[theorem]{Remark}

\DeclareMathOperator{\id}{id}

\DeclareMathOperator{\ad}{ad}

\DeclareMathOperator{\Aut}{Aut}

\DeclareMathOperator{\diag}{diag}

\newcommand{\uq}{\mathbf{U}}
\newcommand{\uqb}{\mathbf{B}}
\newcommand{\uqbns}{\mathbf{B}_{\mathrm{ns}}}

\newcommand{\Ins}{I_{\mathrm{ns}}}

\begin{document}
\title{Characters of Quantum Symmetric Pairs}

\author{Philip Schl\"osser}
\address{Radboud University, IMAPP-Mathematics, Heyendaalseweg 135, 6525 AJ NIJMEGEN, the Netherlands}
\email{philipschloesser98@gmail.com}
\subjclass[2020]{16T20, 17B37}
\keywords{Quantum symmetric pair, iquantum group}

\thanks{This work is funded by grant \texttt{OCENW.M20.108} of the Dutch Research Council NWO
and was inspired by a question asked during the MMRT 2026 conference in Ottawa}

\begin{abstract}
    We classify all characters of quantum symmetric pair coideal subalgebras
    for generalised Satake diagrams of finite and affine type.
\end{abstract}

\maketitle

\section{Introduction}
The theory of quantum symmetric pairs $(\uq,\uqb)$ is a quantisation, a 
$q$-algebra version, of the theory of symmetric pairs 
$(\mathfrak{g},\mathfrak{h})$ (or rather their enveloping algebras) of
Lie algebras.

They were first defined ad hoc in special cases and then studied in general
(say in \cite{Le03} and preceding works) specifically in the context of
computing their zonal spherical functions, which was accomplished in 
\cite{Le04} and \cite{Sch26}.
Nevertheless, also their abstract theory has been receiving
more and more attention, e.g. \cite{Wat23, BK19, KS24}.
In particular, various generalisations become more and more relevant:
the affine case that is included in \cite{Kol14}, and so-called 
pseudo-symmetric cases (\cite{RV20}): cases that correspond to pairs
$(\mathfrak{g},\mathfrak{h})$ of Lie algebras where $\mathfrak{h}$ is no
longer fixed by an involution of $\mathfrak{g}$, and may fail to be
reductive.
Cases which nevertheless can be described using very similar combinatorial
data.

One branch of development of the abstract theory is the representation
theory of $\uqb$: it is usually not a quantum group (a fortiori not a Hopf 
algebra) but what is usually called an $\imath$quantum group.
The representation theory of $\imath$quantum groups is very intractable
and has until now mostly been studied on a case-by-case basis (e.g.
\cite{IK05, AKR17, Wat21, KS24, SW26}, and
ongoing research of Balagovi\'c, Kolb, and Liu) 
and/or by using specialisation arguments (e.g. \cite{Mee26}).

A good first step towards a general understanding can therefore be made
by considering/classifying the 1-dimensional representations, i.e. the
characters (also to illustrate the difficulty of a more general classification).
This has been partially achieved in \cite{Mee26}, where 
(for finite type and symmetric pairs) the integrable characters
are classified, i.e. those that appear as subrepresentations of standard
finite-dimensional modules of $\uq$ and therefore are the only ones relevant
for spherical functions.
However, it is shown in \cite[Example~7.13]{Le00} and more systematically in
\cite[\S3]{IK05} (nonclassical representations with 
$\mathbf{m}_n=\qty(\frac{1}{2},\dots,\frac{1}{2})$) that there exist genuinely 
different characters that are \emph{not}
specialisable.
To the extent of the author's knowledge there does not yet exist a
general classification of \emph{all} characters of $\uqb$ 
except in the particular cases cited above.

In this paper, there are two main theorems that together provide the
classification claimed in the abstract.

The first states that a character $\chi\in\widehat{\uqb}$ is descibed
completely by its restrictions to the torus $Y^\Theta$ (which factors to a smaller
Abelian group $Y^\Theta/2Q_\uqb$)
and to a split $\imath$quantum group
$\uqbns$ (see Section~\ref{sec-qsp} for notation).

\begin{theorem}\label{thm-main-thm}
    Let $\chi$ be a character of $\uqb$. 
    Define $\xi: Y^\Theta\to k^\times$ by
    \[
        \xi(h) := \chi(K_h)\qquad (h\in Y^\Theta).
    \]
    Then $\xi$ descends to a group homomorphism $Y^\Theta/2Q_\uqb\to k^\times$
    and the map
    \[
        \widehat{\uqb}\to (Y^\Theta/2Q_\uqb)^\wedge\times
        \widehat{\uqbns},\qquad
        \chi\mapsto (\xi, \chi|_{\uqbns})
    \]
    is a bijection.
    In particular, $E_i,F_i$ ($i\in I_\bullet$) and
    $B_i$ ($i\in I_\circ\setminus\Ins$) are mapped to zero by every character.
\end{theorem}
\begin{proof}
    The proof is located in Section~\ref{sec-proof}.
\end{proof}

The second main theorem provides a classification of characters for split
$\imath$quantum groups in terms of combinatorial data.

\begin{theorem}\label{thm-char-split}
    Let $(\uq,\uqb)$ be a split quantum symmetric pair with $Y^\Theta=0$,
    and let $\chi$ be the algebra homomorphism $k\langle B_i\mid i\in I\rangle\to k$
    given by $B_i\mapsto b_i$.
    Write $\beta_i:=\frac{b_i^2}{\bm{c}_iq_i}$.

    $\chi$ factors through a character of $\uqb$ if and only
    $(\beta_i)_{i\in I}$ is a split character description (cf.
    Definition~\ref{def-split-char}).
\end{theorem}
\begin{proof}
    The proof is located in Section~\ref{sec-split}.
\end{proof}

In Section~\ref{sec-action}, after proving the two main theorems, we explore 
how to describe the right action of $\widehat{\uq}$, the characters of $\uq$, 
on $\widehat{\uqb}$.

And in Section~\ref{sec-integrable}, we obtain Meereboer's classification
(\cite{Mee26}) of integrable characters by assuming specialisability and
imposing that certain eigenvalues be $q$-numbers.

\section{General Definitions}
\subsection{Root System}
Let $I:= \set{1,\dots,n}$ and let $A=(a_{ij})_{i,j=1,\dots,n}$ be a 
symmetrisable Cartan matrix of finite or affine type, and let $X,Y$ be
two lattices equipped with a perfect pairing 
$\langle\cdot,\cdot\rangle: Y\times X\to\Z$ and with linearly independent
elements
\[
    \alpha_1,\dots,\alpha_n\in X,\qquad h_1,\dots,h_n\in Y
\]
such that
\[
    \forall i,j\in I:\quad \langle h_i,\alpha_j\rangle = a_{i,j}.
\]
Let there be a symmetric form $\cdot$ on $X$ such that
\[
    \forall i,j\in I:\quad \langle h_i,\alpha_j\rangle = 
    2\frac{\alpha_i\cdot\alpha_j}{\alpha_i\cdot\alpha_i}.
\]
This corresponds to a choice $(I,\cdot)$ of Cartan datum (\cite[\S1.1.1]{Lus10})
and a choice $(X,Y)$ of $X$- and $Y$-regular
root datum (\cite[\S2.2.1]{Lus10}) where the embeddings $I\to X,Y$ are
denoted by $i\mapsto\alpha_i,h_i$, respectively.

Furthermore, write $h: X\to Y\otimes_\Z\Q$ for the linear map mapping
$\lambda\in X$ to the unique element $h$ satisfying
\[
    \forall\mu\in X:\quad \lambda\cdot\mu = \langle h,\mu\rangle.
\]
In particular, for $\epsilon_i:= \frac{\alpha_i\cdot\alpha_i}{2}$ we have
\[
    h_{\alpha_i} = \epsilon_i h_i.
\]
Assume that $\epsilon_i\in\Z$, so that $h_{\alpha_i}\in Y$.

Let $W$ be the Weyl group corresponding to $A$; it acts on $X,Y$.

\subsection{Quantum Group}
Let $k$ be a field of characteristic 0 and let $q\in k^\times$ be not a 
root of 1.

For $a,b\in\Z$ and $K$ and invertible element of a $k$-algebra $A$ write
\begin{align*}
    [a]_{q^b} &:= \frac{q^{ab}-q^{-ab}}{q^b-q^{-b}},\\
    [a]_{q^b}^! &:= [a]_{q^b}[a-1]_{q^b}\cdots [1]_{q^b}\quad(a\ge0),\\
    [K;a]_{q^b} &:= \frac{Kq^{ab} - K^{-1}q^{-ab}}{q^b-q^{-b}}
\end{align*}
and $[K]_{q^b}:= [K;0]_{q^b}$ whenever that does not lead to confusion.

Write furthermore $q_i:= q^{\epsilon_i}$ for $i\in I$.

\begin{definition}
    The \emph{Drinfel'd--Jimbo quantum group} $\uq$ is the associative
    $k$-algebra generated by
    \[
        (E_i)_{i\in I},\qquad (F_i)_{i\in I},\qquad (K_h)_{h\in Y}
    \]
    subject to the following relations:
    \begin{alignat*}{2}
        K_0 &= 1 & K_{h}K_{h'}&=K_{h+h'}\\
        K_h E_i &= q^{\langle h,\alpha_i\rangle} E_i K_h &
        K_h F_i &= q^{-\langle h,\alpha_i\rangle} F_i K_h\\
        \comm{E_i}{F_j} &= \delta_{ij} [K_{\alpha_i}]_{q_i}\\
        F_{ij}(E_i,E_j) &= 0\quad (i\ne j)\qquad  &
        F_{ij}(F_i,F_j) &= 0\quad (i\ne j)
    \end{alignat*}
    where
    \[
        F_{ij}(x,y) := \sum_{p+p'=1-a_{ij}}
        \frac{(-1)^{p'}}{[p]^!_{q_i} [p']^!_{q_i}} x^p y x^{p'},
    \]
    $h,h'\in Y, i,j\in I$, and where $K_\mu$ is understood to mean
    $K_{h_\mu}$ for $\mu\in X$ with $h_\mu\in Y$.
\end{definition}

There is a well-established way (e.g. \cite[\S3.3.4]{Lus10}) of turning $\uq$
into a Hopf algebra with comultiplication $\Delta$, counit $\epsilon$,
and antipode $S$.
Sweedler notation will be used whenever appropriate.

For any $\uq$-module $V$ and $\mu\in X$ let
\[
    V_\mu:=\set{v\in V\where\forall h\in Y:\quad K_hv = q^{\langle h,\mu\rangle} v},
\]
the \emph{$\mu$-weight space}.
This is especially applicable to the case where $V=\uq$ via the adjoint action
\[
    \ad(x)(y) := x_{(1)} y S(x_{(2)})\qquad (x,y\in\uq).
\]

For $w\in W$ write $T_w$ for the automorphism $T''_{w,1}:\uq\to\uq$
from \cite[\S37.1.3,\S39.4.7]{Lus10}. It satisfies $T_w(K_h)=K_{wh}$ for
$h\in Y$.

For a subset $J\subset I$, consider the associated \emph{Levi subalgebra}
$\uq_J$, i.e. the Hopf subalgebra of $\uq$ generated by
\[
    (E_i)_{i\in J},\qquad (F_i)_{i\in J},\qquad (K_h)_{h\in Y}.
\]

\subsection{Quantum Symmetric Pairs}\label{sec-qsp}
Let $(I_\bullet,\tau)$ be a generalised Satake diagram (see \cite[\S2.2]{RV20}
for a more extensive discussion),
i.e.
\begin{enumerate}
    \item $I_\bullet\subset I$ is a subdiagram of finite type (\enquote{black nodes});
    \item $\tau$ is a diagram involution of $I$ whose restriction to $I_\bullet$
        is minus the longest element of the Weyl group of $I_\bullet$;
    \item for every $i\in I_\circ:= I\setminus I_\bullet$ (\enquote{white nodes}),
        the associated rank 1 diagram $I[i]$ (all nodes in $I_\bullet\cup\set{i,\tau(i)}$
        that are connected to $i$ or $\tau(i)$) is not of type $\mathsf{A}_2$.
\end{enumerate}
We assume in addition that there are involutions of $X,Y$ compatible with
$\tau$.
Let $w_\bullet\in W$ be the longest element of the Weyl group $W_\bullet\le W$
associated with $I_\bullet$.
Set
\begin{align*}
    \Theta&:= -w_\bullet\circ\tau\in\Aut(X),\Aut(Y),\\
    \Ins &:= \set{i\in I_\circ\where \tau(i)=i,
    \forall j\in I_\bullet: \langle h_i,\alpha_j\rangle=0},\\
    \mathcal{S}&:= \set{i\in\Ins\where\forall j\in\Ins: 2|\langle h_j,\alpha_i\rangle}.
\end{align*}
$\Ins$ can alternatively be characterised as $i\in I_\circ$ such that
$I[i]$ is of type $\mathsf{AI}_1$ and is called the \emph{split subdiagram}.
$\mathcal{S}$ contains the \emph{non-standard nodes}.

\begin{definition}
    A generalised Satake diagram $(I_\bullet,\tau)$ is said to be
    \emph{quasi-split} if $I_\bullet=\emptyset$ and \emph{split} if in addition
    $\tau=\id$.
\end{definition}

Define $Q_\uqb$ as the sublattice of $Y$ generated by
$h_{\alpha_i}$ for $i\in I_\bullet$ and $h_{\alpha_i-\alpha_{\tau(i)}}$
for $i$ such that $I[i]$ is of type $\diag(\mathsf{A}_1)$.

Write $\uq_\bullet$ for the subalgebra of $\uq$ generated by
\[
    (E_i)_{i\in I_\bullet},\quad (F_i)_{i\in I_\bullet},\quad
    (K_{\pm\alpha_i})_{i\in I_\bullet}.
\]
Let $\uqb=\uqb_{\bm{c},\bm{s}}$ be the coideal subalgebra for $(I_\bullet,\tau)$
($\imath$quantum group) with
parameters $\bm{c}=(\bm{c}_i)_{i\in I_\circ}\in (k^\times)^{I_\circ}$ and
$\bm{s}=(\bm{s}_i)_{i\in I_\circ}\in k^{I_\circ}$.
It is generated by $\uq_\bullet$, $k[Y^\Theta]$ (i.e. $K_h$ for $h\in Y^\Theta$),
and
\[
    B_i := F_i + \bm{c}_i T_{w_\bullet} (E_{\tau(i)})K_{-\alpha}
    + \bm{s}_i K_{-\alpha_i}
\]
for $i\in I_\circ$.
We assume that the parameters satisfy the following conditions:
\begin{enumerate}
    \item if $I[i]$ is of type $\diag(\mathsf{A}_1)$: $\bm{c}_i=\bm{c}_{\tau(i)}$;
    \item if $i\not\in\mathcal{S}$: $\bm{s}_i=0$.
\end{enumerate}
$(\uq,\uqb)$ is the \emph{generalised quantum symmetric pair} for the diagram
$(I_\bullet,\tau)$.

For a generalised Satake subdiagram $J\subset I$ (i.e. $\tau$-invariant subset
such that $\forall i\in J\cap I_\circ: I[i]\subset J$), let $\uqb_J$ be
the subalgebra of $\uq_J$ generated by
\[
    (E_i)_{i\in J\cap I_\bullet},\qquad
    (F_i)_{i\in J\cap I_\bullet},\qquad
    (K_h)_{h\in Y^\Theta},\qquad
    (B_i)_{i\in J\cap I_\circ}.
\]
Then $(\uq_J,\uqb_J)$ is the generalised quantum symmetric pair associated to $(J,\tau)$.

In particular, write $\uqbns$ for the \emph{split part} of $\uq$, i.e. the
subalgebra of $\uq_{\Ins}$ generated by the $(B_i)_{i\in\Ins}$.
Then $\uqb_{\Ins} \cong k[Y^\Theta]\otimes \uqbns$ as $k$-vector spaces, where
both tensor components commute with one another.
Theorem~\ref{thm-main-thm} then states that restriction to $\uqb_{\Ins}$ is
an injection whose image consists of those characters that vanish on
$K_\alpha-1$ for $\alpha\in 2Q_\uqb$.

\section{Characters of $\uq$}
We start by recapping the classification of the characters of $\uq$
(or any Levi subalgebra thereof).

\begin{definition}
    Write $Q$ for the sublattice of $Y$ generated by $\epsilon_i h_i$ for
    $i\in I$ (i.e. the image of the root lattice under the embedding
    $h$).
\end{definition}

\begin{lemma}\label{lem-U-lattice-descent}
    Let $\chi:\uq\to k$ be a character of $\uq$, and let $\xi: Y\to k^\times$ be
    given by $\xi(h) := \chi(K_h)$. Then $\xi$ is a group homomorphism that
    factors through $Y/2Q$.
\end{lemma}
\begin{proof}
    For $h_1,h_2\in Y$ we have
    \[
	\xi(h_1+h_2)=\chi(K_{h_1+h_2})
	= \chi(K_{h_1} K_{h_2})
	= \chi(K_{h_1})\chi(K_{h_2})
	= \xi(h_1)\xi(h_2).
    \]
    Furthermore, for each $i\in I$, applying $\chi$ to the relation
    \[
        E_iF_i - F_iE_i = [K_{\alpha_i}]_{q_i}
    \]
    of $\uq$ yields $0 = \xi(\epsilon_ih_i)-\xi(-\epsilon_ih_i)$,
    which implies that $\xi$ vanishes on $2Q$.
\end{proof}

\begin{lemma}\label{lem-U-weight-spaces}
    Let $\mu\ne0$ and let $x\in\uq_\mu$, then $\chi(x)=0$.
\end{lemma}
\begin{proof}
    Let $h\in Y$ with $\langle h,\mu\rangle=0$ (exists, as the pairing of
    $X,Y$ is non-degenerate), then
    \[
	K_h x K_h^{-1} = q^{\langle h,\mu\rangle} x.
    \]
    Applying $\chi$ and reordering yields
    $(1-q^{\langle h,\mu\rangle}) \chi(x)=0$,
    which implies $x=0$.
\end{proof}

\begin{proposition}\label{prop-char-U}
    The map $\chi\mapsto\xi$ where $\xi(h) := \chi(K_h)$ ($h\in Y$) establishes
    a bijection between the sets of characters of $\uq$ on one hand, and the
    set of characters of the Abelian group $Y/2Q$ on the other hand.
\end{proposition}
\begin{proof}
    For well-definedness note that by Lemma~\ref{lem-U-lattice-descent},
    any $\xi$ obtained by restricting a character acts trivially on $2Q$.

    For injectivity, let $\chi,\chi'$ be two characters of $\uq$ that
    restrict to the same character $\xi$ of $Y/2Q$.
    By Lemma~\ref{lem-U-weight-spaces}, they both map $E_i,F_i$ ($i\in I$)
    to 0.
    Since $\uq$ is generated by $K_h, E_i, F_i$ ($h\in Y, i\in I$), we
    have $\chi=\chi'$.

    For surjectivity, given $\xi$, we define
    \[
	\chi(K_h) := \xi(h),\qquad \chi(E_i):= 0,\qquad \chi(F_i):= 0
    \]
    and note that all relations between these generators of $\uq$ are mapped
    to 0.
\end{proof}

\section{Proof of Theorem~\ref{thm-main-thm}}\label{sec-proof}
We start by proving well-definedness.
\begin{lemma}\label{lem-xi-factors}
    Let $\chi\in\widehat{\uqb}$, let $\xi: Y^\Theta\to k^\times$ be defined
    by having $h\mapsto \chi(K_h)$.
    Then $\xi$ induces a group homomorphism $Y^\Theta/2Q_\uqb\to k^\times$.
\end{lemma}
\begin{proof}
    Note that
    \[
        \xi(h_1+h_2) = \chi(K_{h_1+h_2})
        = \chi(K_{h_1})\chi(K_{h_2})
        = \xi(h_1)\xi(h_2)
    \]
    for $h_1,h_2\in Y^\Theta$,
    so $\xi$ is indeed a group homomorphism.

    Let $i\in I_\bullet$, then as in the proof of Lemma~\ref{lem-U-lattice-descent} we have
    \[
	    \xi(2\epsilon_ih_i)=1.
    \]
    Let now $i\in I_\circ$ such that $I[i]$ is of type $\diag(\mathsf{A}_1)$,
    then
    \[
        B_i = F_i + c_i E_{\tau(i)} K_i^{-1},\qquad
        B_{\tau(i)} = F_{\tau(i)} + c_i E_i K_{\tau(i)}^{-1}.
    \]
    Furthermore, we have
    \[
        \Delta(B_i) = B_i\otimes K_i^{-1} + 1\otimes F_i + c_i K_{\tau(i)} K_i^{-1}\otimes E_{\tau(i)} K_i^{-1}
    \]
    (analogously for $\Delta(B_{\tau(i)})$), so by \cite[Equation~(7.5)]{Kol14}
    we have $\mathcal{Z}_i = K_{\tau(i)}K_i^{-1}$.
    Then we have
    \[
	\comm{B_i}{B_{\tau(i)}} = c_i [\mathcal{Z}_i]_{q_i}.
    \]
    Applying $\chi$ and multiplying by a suitable nonzero scalar yields
    $0 = \xi(\alpha_i-\alpha_{\tau(i)})-\xi(\alpha_{\tau(i)}-\alpha_i)$,
    which shows that $\xi(2\alpha_i-2\alpha_{\tau(i)})=1$.
\end{proof}

The next statement all but shows injectivity.
\begin{lemma}\label{lem-stuff-mapped-to-0}
    Let $\chi\in\widehat{\uqb}$.
    \begin{enumerate}
        \item For $i\in I_\bullet$ we have
            $\chi(E_i)=\chi(F_i)=0$.
        \item If $i\in I_\circ$ such that $I[i]$ is not quasi-split
            (i.e. has black dots), then $\chi(B_i)=0$.
        \item If $i\in I_\circ$ such that $I[i]$ is quasi-split but not split
            (i.e. $\tau|_{I[i]}$ is non-trivial), then $\chi(B_i)=0$.
    \end{enumerate}
\end{lemma}
\begin{proof}
    For every element $x\in\uqb$ we are going to find $h\in Y^\Theta$ such
    that $K_h x K_{-h} = q^a x$ for $a\ne0$:
    \begin{enumerate}
        \item Take $h:=h_i\in Y^\Theta$, then $a=\pm 2$.
        \item In this case there exists $j\in I_\bullet$ such that $\langle h_j,\alpha_i\rangle\ne0$,
        then $h:=h_j\in Y^\Theta$ and $a=-\langle h_j,\alpha_i\rangle$.
        \item In this case take $h:= h_i-h_{\tau(i)}$ (non-zero, as
        $i\ne\tau(i)$), then $a=-2+\langle h_{\tau(i)},\alpha_i\rangle$
        (non-zero as $\langle h_{\tau(i)},\alpha_i\rangle\le0$).
    \end{enumerate}
    Applying $\chi$ to this equation then gives us
    \[
        (1-q^a)\chi(x) = 0.
    \]
    As we assume that $q$ is not a root of unity, we conclude $\chi(x)=0$.
\end{proof}

Next, we want to prove surjectivity, which involves proving Serre-like
relations.
These often involve an element $\mathcal{Z}_i$.

\begin{definition}
    For $i\in I_\circ$, define $\mathcal{Z}_i\in \uqb_{I_\bullet}$
    as
    \[
        \mathcal{Z}_i:= \ad(x)(K_{2\alpha_{\tau(i)}})K_{-\alpha_i-\alpha_{\tau(i)}},
    \]
    where $x\in\uq_\bullet$ such that 
    $T_{w_\bullet}(E_{\tau(i)})=\ad(x)(E_{\tau(i)})$.
\end{definition}

\begin{lemma}\label{lem-Z-annihilate}
    If $I[i]$ is not quasi-split (i.e. has black nodes), then any character
    of $\uqb_{I_\bullet}$ annihilates $\mathcal{Z}_i$.
\end{lemma}
\begin{proof}
    We have $w_\bullet(\alpha_{\tau(i)})\ne \alpha_{\tau(i)}$, so that
    $w_\bullet(\alpha_{\tau(i)})-\alpha_{\tau(i)}$ is a nonzero weight.
    It just so happens that $\mathcal{Z}_i$ has that weight.
    By Lemma~\ref{lem-U-weight-spaces}, every character of
    $\uqb_{I_\bullet}$ therefore has to annihilate $\mathcal{Z}_i$.
    Note that since $I_\bullet$ has no white nodes, and since
    $Y^\Theta$ contains $h_i$ for $i\in I_\bullet$, $\uqb_{I_\bullet}$ is
    a quantum group, so the lemma applies.
\end{proof}

\begin{proposition}\label{prop-surj}
    Let $\xi: Y^\Theta/2Q_\uqb\to k^\times$ be a group homomorphism,
    and $\eta\in\widehat{\uqbns}$.
    There is a character $\chi\in\widehat{\uqb}$ that restricts to $\xi$ and
    $\eta$.
\end{proposition}
\begin{proof}
    Consider the quantum symmetric pair $(\uq',\uqb'):=
    (\uq_{I\setminus\Ins},\uqb_{I\setminus\Ins})$.

    We extend $\xi$ arbitrarily to $Y$.
    Since $2Q'\le 2Q_\uqb$ ($Q'$ being the root lattice of $I\setminus\Ins$,
    embedded into $Y$), there exists a character $\chi'$ of the
    quantum group $\uq'$ that restricts to $\xi$ on $Y^\Theta$
    by Proposition~\ref{prop-char-U}.
    We can then restrict $\chi'$ to $\uqb'$, where it maps all $B_i$ to 0 
    ($i\not\in\Ins$).

    Let $A$ be the (free) algebra generated by $\uqb'$ and
    $\uqbns$, and define its character $\chi$ by having
    \[
        x\in \uqb'\mapsto \chi'(x),\qquad
        x\in \uqbns \mapsto \eta(x).
    \]
    Our goal now is to show that $\chi$ in fact factors through $\uqb$, which
    is a quotient of $A$ by the following relations:
    \begin{align*}
        K_h B_j &= B_j K_h\qquad (h\in Y^\Theta, j\in\Ins)\\
        E_i B_j &= B_j E_i\qquad (i\in I_\bullet, j\in\Ins)\\
        F_i B_j &= B_j F_i\qquad (i\in I_\bullet, j\in\Ins)\\
        F_{ij}(B_i,B_j) &= C_{ij} \qquad (i\in I_\circ\setminus\Ins, j\in\Ins,
        \text{and vice-versa})
    \end{align*}
    (see \cite[Theorem~7.4]{Kol14}).
    Note that the first three conditions are trivially satisfied because
    $k$ is commutative, so it suffices to consider the Serre relations.
    Note that as one of $\chi(B_i),\chi(B_j)$ is always zero, the left-hand side,
    $\chi(F_{ij}(B_i,B_j))$ is zero, so that it suffices to show that
    $\chi(C_{ij})=0$.

    We first consider the case where $I[i]$ is non-split, and $I[j]$ is split,
    then $i\ne\tau(j)$, so we only need to consider the $\delta_{i,\tau(i)}$
    term of $C_{ij}$ and hence assume without loss of generality that
    $\tau(i)=i$ (otherwise $C_{ij}=0$).
    This means in particular, that $I[i]$ cannot be quasi-split (as it would
    be split otherwise), so that $\chi(\mathcal{Z}_i)=0$ by 
    Lemma~\ref{lem-Z-annihilate}.
    By \cite[Theorem~7.4]{Kol14}, for $\langle h_i,\alpha_j\rangle=0,-1,-2$ 
    (the only three possibilities), we have 
    $\chi(C_{ij})\propto\chi(\mathcal{Z}_i)=0$, which establishes the
    Serre relation.

    If conversely $I[i]$ is split and $I[j]$ is not, we also have $i\ne\tau(j)$,
    so that we can still focus on the $\delta_{i,\tau(i)}$ terms.
    In this case, we have $\chi(C_{ij})\propto\chi(B_j)=0$, which establishes
    the other Serre relation.
\end{proof}

\begin{proof}[{Proof of Theorem~\ref{thm-main-thm}}]
    For well-definedness, note that by Lemma~\ref{lem-xi-factors}, $\xi$
    is a group homomorphism that factors through $Y^\theta/2Q_\uqb$.
    Moreover, since $\chi|_{\uqbns}$ is the restriction of an algebra 
    homomorphism to a subalgebra, it is still an algebra homomorphism.

    Let now $\chi_1,\chi_2\in\widehat{\uqb}$ restrict to the same
    group homomorphism $\xi: Y^\theta/2Q_\uqb\to k^\times$ and
    the same character of $\uqbns$.
    We know that $\uqb$ is generated by $k[Y^\Theta]$, by $\uq_\bullet$,
    and by the $B_i$ ($i\in I_\circ$).
    By assumption we therefore already know that $\chi_1,\chi_2$ coincide on
    all generators except $E_i,F_i$ ($i\in I_\bullet$) and
    $B_i$ ($i\in I_\circ\setminus\Ins$).
    By Lemma~\ref{lem-stuff-mapped-to-0} the remaining generators are mapped
    to 0 by $\chi_1,\chi_2$.
    Consequently, $\chi_1=\chi_2$, which proves injectivity.

    Surjectivity follows from Proposition~\ref{prop-surj}.
\end{proof}

We see that the characters of $\uqb$ can be described fully by characters
of the Abelian group $Y^\Theta/2Q_\uqb$ and by characters of the
split part $\uqbns$ of $\uqb$.

We shall now proceed to classify the characters of split quantum symmetric pairs.

\section{Characters for Split QSPs}\label{sec-split}
We now assume that $(\uq,\uqb)$ is split and that $Y^\Theta=0$.
Now, $\uqb$ is generated by $(B_i)_{i\in\Ins}$ subject to the Serre relations 
from \cite[Equation~3.11f.]{CLW20}.

As every Serre relation only concerns two generators, we first consider the
case where $I$ is of rank 2.

\begin{lemma}\label{lem-necessary-serre-2}
    Assume $I=\set{1,2}$, say $\abs{\alpha_1}^2\ge\abs{\alpha_2}^2$.
    Let $b_1,b_2\in k$ and define $\chi: k\langle B_1,B_2\rangle\to k$
    by mapping $B_i\mapsto b_i$ ($i=1,2$).
    Write $\beta_i:=\frac{b_i^2}{\bm{c}_iq_i}$ and define the
    two polynomials (depending on the length of $\alpha_i$):
    \begin{align*}
        f_{-1,i}(x) &:= ([2]_{q_i}-2)x+1 = (q_i^{1/2}-q_i^{-1/2})^2x+1\\
        f_{-3,i}(x) &:= (q_i^{1/2}-q_i^{-1/2})^2(q_i^{3/2}-q_i^{-3/2})^2x^2
        \\&\qquad - (q_i^{3/2}-q_i^{-3/2})^2(q_i^2-2q_i+4-2q_i^{-1}+q_i^{-2})x\\
        &\qquad+ [3]_{q_i}^2.
    \end{align*}
    $\chi$ defines a character of $\uqb$ if and only if:
    \begin{description}
        \item[Type $\mathsf{A}_1\times\mathsf{A}_1$ or $\mathsf{A}_1^{(1)}$]
        no further conditions;
        \item[Type $\mathsf{A}_2$] $\beta_i=\beta_j=0$ or
        $f_{-1,i}(\beta_i)=f_{-1,j}(\beta_j)=0$;
        \item[Type $\mathsf{B}_2$ or
        $\mathsf{A}_2^{(2)}$] $\beta_j=0$ or $f_{-1,i}(\beta_i)=0$;
        \item[Type $\mathsf{G}_2$] $\beta_i=\beta_j=0$ or 
        $f_{-1,i}(\beta_i)=f_{-3,j}(\beta_j)=0$.
    \end{description}
\end{lemma}
\begin{proof}
    Let $(i,j)=(1,2)$ or $(2,1)$.
    Write $a:= \langle h_i,\alpha_j\rangle$,
    then by \cite[Equation~3.11f.]{CLW20} and application of $\chi$, 
    we obtain the following Serre relations for $(i,j)$:
    no further conditions for $a$ even,
    \[
        (2-[2]_{q_i})b_i^2b_j =q_i \bm{c}_i b_j
    \]
    for $a=-1$, and
    \begin{align*}
        (2-2[4]_{q_i}+[2]_{q_i^2}[3]_{q_i}) b_i^4b_j &= 
        -[2]_{q_i} ([2]_{q_i}[4]_{q_i} + q_i^2 + q_i^{-2})
        q_ic_i b_i^2b_j\\
        &\qquad + 2([3]_{q_i}^2+1)q_ic_ib_i^2b_j\\
        &\qquad- [3]_{q_i}^2 q_i^2c_i^2 b_j
    \end{align*}
    for $a=-3$.
    Evidently, only cases where $a$ is odd are of interest.
    These cases can be rewritten to
    $b_j f_{-1,i}(\beta_i)=0$ ($a=-1$) and
    $b_j f_{-3,i}(\beta_i)=0$ ($a=-3$).

    $\chi$ factors through $\uqb$ if and only if the above relations are
    satisfied for $(1,2)$ and $(2,1$).
    \begin{description}
        \item[Type $\mathsf{A}_1\times\mathsf{A}_1$ or $\mathsf{A}_1^{(1)}$]
            The relations are always satisfied.
        \item[Type $\mathsf{A}_2$] The relations read
            \[
                b_1 f_{-1,2}(\beta_2) = b_2 f_{-1,1}(\beta_1)=0.
            \]
            As $f_{-1,i}(0)\ne0$ for $i=1,2$, the above is equivalent to
            $\beta_1=\beta_2=0$ or $f_{-1,1}(\beta_1)=f_{-1,2}(\beta_2)=0$.
        \item[Type $\mathsf{B}_2$ or $\mathsf{A}_2^{(2)}$] There is only one
        interesting relation:
        \[
            b_2 f_{-1,1}(\beta_1)=0,
        \]
        which is true satisfied if and only if $\beta_2=0$ or
        $f_{-1,1}(\beta_1)=0$.
        \item[Type $\mathsf{G}_2$] The relations read
        \[
            b_2 f_{-1,1}(\beta_1) = b_1 f_{-3,2}(\beta_2) = 0,
        \]
        which is equivalent to $\beta_1=\beta_2=0$ or
        $f_{-1,1}(\beta_1)=f_{-3,2}(\beta_2)=0$ as neither $f_{-1}$ nor
        $f_{-3}$ have 0 as root.
    \end{description}
\end{proof}

Recursion of these conditions for general rank suggest the following definition:

\begin{definition}\label{def-split-char}
    Write $\Gamma$ for the subgraph of $I$ where all higher-degree edges
    are removed.
    Note that on every connected component $A$ of $\Gamma$, the root length
    $\abs{\alpha}$ is constant, so in particular we can define
    $f_{-1,A},f_{-3,A}$.

    A family $(\beta_i)_{i\in I}$ is called a \emph{split character description}
    if the following conditions are satisfied:
    \begin{enumerate}
        \item for every $i\in I$, $\beta_i\bm{c}_iq_i$ has a square root in $k$;
        \item on every connected component $A$ of $\Gamma$, 
        $\beta$ is constant, say $\beta_A$;
        \item for every connected component $A$ of $\Gamma$ of rank $>1$, we
        have $\beta_A=0$ or $f_{-1,A}(\beta_A)=0$;
        \item whenever two connected components $A_1,A_2$ of $\Gamma$ are 
            connected in $I$ by
        a directed double or quadruple edge from $A_1$ to $A_2$, we have
        $\beta_{A_2}=0$ or $f_{-1,A_1}(\beta_{A_1})=0$.
        \item whenever two connected components $A_1,A_2$ of $\Gamma$ are 
            connected in $I$ by a triple edge from $A_1$ to $A_2$, we have 
            $\beta_{A_1}=\beta{A_2}=0$ or
        $f_{-1,A_1}(\beta_{A_1})=f_{-3,A_2}(\beta_{A_2})=0$.
    \end{enumerate}
\end{definition}

\begin{proof}[{Proof of Theorem~\ref{thm-char-split}}]
    Assume $\chi$ factors through a character of $\uqb$.
    By definition of $\beta_i$, the element $\beta_ic_iq_i$ has a square root
    (for all $i\in I$).
    
    Let $A\subset\Gamma$ be a connected component.
    In particular, $A$ is path connected.
    Let $i,j\in A$ and let $i=i_0,\dots,i_r=j$ be a path connecting $i$ with
    $j$.
    If $i=j$, then $\beta_i=\beta_j$ evidently.
    Otherwise, we apply Lemma~\ref{lem-necessary-serre-2} to any $\set{i_{k-1},i_k}$
    (which is of type $\mathsf{A}_2$) and obtain that $\beta_i=\beta_j$
    is either 0 or the root of $f_{-1,A}$.
    This shows that $\beta_i=\beta_j$, hence that $\beta$ is constant on $A$
    and also restricts the values that $\beta_A$ can take in case $\# A>1$.

    The remaining two conditions follow by direct application of 
    Lemma~\ref{lem-necessary-serre-2} to the two nodes adjacent to the
    double, triple, or quadruple edge.

    For the converse let $i,j\in I$ be distinct (without loss of generality
    $\abs{\alpha}_i\ge\abs{\alpha_j}$).
    Depending on the type of diagram that $\set{i,j}$ is, we can conclude the
    following about $\beta_i,\beta_j$:
    \begin{description}
        \item[Type $\mathsf{A}_1\times\mathsf{A}_1$ or $\mathsf{A}_1^{(1)}$]
        no further conditions;
        \item[Type $\mathsf{A}_2$] Then $i,j$ belong to the same connected
        component $A$ (with more than one element) of $\Gamma$.
        We have $\beta_i=\beta_j=0$ or
        $f_{-1,i}(\beta_i)=f_{-1,j}(\beta_j)=0$.
        \item[Type $\mathsf{B}_2$ or $\mathsf{A}_2^{(2)}$] $i$ and $j$ belong
        to connected components in $\Gamma$ that are connected by a
        double or quadruple edge.
        Then $\beta_j=0$ or $f_{-1,i}(\beta_i)=0$;
        \item[Type $\mathsf{G}_2$] $i$ and $j$ belong to connected components
        in $\Gamma$ that are connected by a triple edge.
        Then $\beta_i=\beta_j=0$ or 
        $f_{-1,i}(\beta_i)=f_{-3,j}(\beta_j)=0$.
    \end{description}
    This implies by Lemma~\ref{lem-necessary-serre-2} that $\chi$ preserves
    the Serre relations between $i$ and $j$.
    We conclude that $\chi$ preserves all Serre relations and hence defines
    a character of $\uqb$.
\end{proof}

\begin{remark}
    Theorems~\ref{thm-main-thm} and \ref{thm-char-split} conclude the
    classification of characters.
    To summarise: a character of $\uqb$ for any generalised quantum affine
    pair $(\uq,\uqb)$ of finite or affine type is given by
    \begin{enumerate}
        \item A character $\xi$ of the Abelian group $Y^\Theta/2Q_\uqb$ and
        \item an appropriate choice of square roots for a split character
        description $(\beta_i)_{i\in\Ins}$ of the split subdiagram $\Ins$.
    \end{enumerate}
\end{remark}

\begin{example}
    For $n>1$, the $\imath$quantum group of type $\mathsf{CI}_n$ has the
    following characters (if $k$ is algebraically closed):
    \begin{enumerate}
        \item The classical characters: $B_1,\dots,B_{n-1}\mapsto 0$ and
            $B_n$ mapped to an arbitrary element of $k$;
        \item $2^{n-1}$ non-classical characters mapping $B_n$ to 0 and
            each $B_i$ to a square root of $\frac{-\bm{c}_iq_i}{q_i+q_i^{-1}-2}$
            for $i=1,\dots,n-1$.
    \end{enumerate}

    For quantum symmetric pairs of finite type, unless $(\uq,\uqb)$ is of type 
    $\mathsf{AI}$, $\mathsf{AIII}_{2p-1,p}$,
    $\mathsf{BI}$, $\mathsf{CI}$, $\mathsf{DI}$, $\mathsf{DIII}_{\mathrm{even}}$,
    $\mathsf{EI}$, $\mathsf{EII}$, $\mathsf{EV}$, \dots, $\mathsf{EIX}$,
    $\mathsf{FI}$, $\mathsf{G}$, characters of $\uqb$ are entirely given by
    characters of $Y^\Theta/2Q_\uqb$.
\end{example}

\section{Action of \texorpdfstring{$\widehat{\uq}$}{Uhat}}\label{sec-action}
Since the category of $\uqb$-modules is a right tensor category over the
category of $\uq$-modules, we obtain a right action of the group $\widehat{\uq}$
on the set $\widehat{\uqb}$.

\begin{proposition}\label{prop-character-action}
    Let $\chi_2\in\widehat{\uq}$ be described by $\xi_2: Y/2Q\to k^\times$
    as in Proposition~\ref{prop-char-U} and let $\chi_1\in\widehat{\uqb}$
    map
    \[
        K_h\mapsto \xi_1(h),\qquad B_i\mapsto b_i
    \]
    for $h\in Y^\Theta$, $\xi_1: Y^\Theta/2Q_\uqb\to k^\times$, and
    $i\in\Ins$.
    Then $\chi_1\chi_2$ maps
    \[
        K_h\mapsto (\xi_1\xi_2)(h),\qquad
        B_i\mapsto \xi_2(\alpha_i)b_i
    \]
    for $h\in Y^\Theta$ and $i\in\Ins$.
\end{proposition}
\begin{proof}
    Note that $(\chi_1\chi_2)(x) = \chi_1(x_{(1)})\chi_2(x_{(2)})$ for
    $x\in\uqb$.

    Since $\uqb$ is a right coideal and since the comultiplication $\Delta$ is
    an algebra homomorphism, $\chi_1\chi_2$ is a well-defined character
    of $\uqb$.

    For the description, note that
    \[
        (\chi_1\chi_2)(K_h) = \chi_1(K_h)\chi_2(K_h) = \xi_1(h)\xi_2(h)
    \]
    for $h\in Y^\Theta$.
    Furthermore, for $i\in I_\bullet$ we have
    \begin{align*}
        (\chi_1\chi_2)(E_i) &= \chi_1(E_i) + \chi_1(K_i)\chi_2(E_i) = 0\\
        (\chi_1\chi_2)(F_i) &= \chi_1(F_i)\chi_2(K_i)^{-1} + \chi_2(F_i) = 0,
    \end{align*}
    Furthermore, for $i\in\Ins$ we have
    \begin{align*}
        (\chi_1\chi_2)(B_i) &= \chi_1(B_i)\chi_2(K_i)^{-1}
        + \chi_2(F_i) + c_i \chi_2(E_iK_i^{-1})
        = \xi_2(-\alpha_i)b_i.
    \end{align*}
    Moreover, since $2\alpha_i\in 2Q$, we have $\xi_2(\alpha_i)=\xi_2(-\alpha_i)$
    by Lemma~\ref{lem-U-lattice-descent}.
\end{proof}

\begin{remark}
    Within our description of characters of $\uqb$, the action of $\widehat{\uq}$
    changes the character of $Y^\Theta/2Q_\uqb$ by multiplication, and
    may change the choices of square roots of the split character description.
\end{remark}

\subsection{Integrable Characters}\label{sec-integrable}
The integrable characters for Satake diagrams of finite type have been
classified in \cite{Mee26}.
We shall now recover this classification by imposing a few na\"ive conditions
on the values.
To make our lives easier, we assume the additional conditions of
\cite[Equation~3.1ff.]{Wat23}, in particular that $\bm{c}_i=q_i^{-1}$ for
$i\in\Ins$.
Moreover, we assume, as in \cite[Remark 3.3.9]{Wat23} that
$\bm{s}_i = [s_i]_{q_i}$ for $i\in\mathcal{S}$ and $s_i\in\Z$.
These parameters are in particular specialisable.

The big classification result of Meereboer (\cite{Mee26})
is as follows: integrable characters map
\[
    K_h\mapsto q^{\langle h,\lambda\rangle},\qquad
    B_i\mapsto [s_i+n_i]_{q_i}
\]
for $h\in Y^\Theta, i\in\mathcal{S}$, $\lambda\in X$ dominant, and
$n_i\in\Z$ where $\abs{n_i}\in\langle\lambda,\alpha_i\rangle-2\N_0$, 
and all other generators to 0.
Specialisation plays an important role here, as part of the result is that
specialisation is bijective.

Since we do not assume anything about $q$ except that it is not a root of unity,
we cannot rigorously proceed with specialisation arguments, but considering
the expressions for $f_{-1,i}$ and $f_{-3,i}$, we quickly see that any 
reasonable attempt at specialisation makes them constant nonzero polynomials.
In lieu of requiring specialisability, we can therefore consider what happens
to our definition of a split character description if we eliminate the
possibility of $f_{-1,A}(\beta_A)=0$ or $f_{-3,A}(\beta_A)=0$:
\begin{enumerate}
    \item For every $i\in I$, $\beta_i\bm{c}_iq_i$ has a square root in $k$;
    \item on every connected component $A$ of $\Gamma$, $\beta$ is constant,
        say $\beta_A$;
    \item for every connected component $A$ of $\Gamma$ of rank $>1$,
        $\beta_A=0$;
    \item whenever two connected components $A_1,A_2$ are connected by a
        directed double or quadruple edge from $A_1$ to $A_2$, we have
        $\beta_{A_2}=0$;
    \item whenever two connected components $A_1,A_2$ are connected by a
        directed triple edge from $A_1$ to $A_2$, we have 
        $\beta_{A_1}=\beta_{A_2}=0$.
\end{enumerate}
This simplifies to the following rules for $(\beta_i)_{i\in\Ins}$:
$\beta_i=0$ unless $2|\langle h_j,\alpha_i\rangle$ for
all $j\in\Ins$, i.e. $\beta_i=0$ unless $i\in\mathcal{S}$.

Moreover, from \cite[Proposition~3.3.4, Remark~3.3.9]{Wat23}, we see
that in a one-dimensional quotient of $L(\lambda)$ (the standard
$\uq$-module with highest weight $\lambda$) as $\uqb$-modules,
we can only have $K_h$ acting as $q^{\langle h,\lambda\rangle}$ and
$B_i$ as $[s_i+n_i]_{q_i}$ with $\abs{n_i}\in\langle\lambda,\alpha_i\rangle - 2\N_0$,
which gives us precisely Meereboer's classification.

\printbibliography
\end{document}